\documentclass[a4paper,12pt]{amsart}
\usepackage{amssymb}
\usepackage{amsmath, amsthm, amscd, amsfonts, amssymb, graphicx, color}
\usepackage{amsmath, amsthm, amscd, amsfonts, amssymb, graphicx, color}

\newtheorem{theorem}{Theorem}[section]
\newtheorem{lemma}[theorem]{Lemma}
\newtheorem{proposition}[theorem]{Proposition}
\newtheorem{corollary}[theorem]{Corollary}
\theoremstyle{definition}
\newtheorem{definition}[theorem]{Definition}
\newtheorem{example}[theorem]{Example}

\theoremstyle{remark}
\newtheorem{remark}[theorem]{Remark}
\numberwithin{equation}{section}

\title[Disjoint Linear Dynamical Properties of Elementary Operators]
{Disjoint Linear Dynamical Properties of Elementary Operators}

\author[S. Ivkovi\'{c}]{Stefan Ivkovi\'{c}}
\address{Mathematical Institute of the Serbian Academy of Sciences and Arts,
	p.p. 367, Kneza Mihaila 36, 11000 Beograd, Serbia}
\email{stefan.iv10@outlook.com}

\author[S.M. Tabatabaie]{Seyyed Mohammad Tabatabaie}
\address{Department of Mathematics, University of Qom, Qom, Iran}
\email{sm.tabatabaie@qom.ac.ir}



\date{\today}

\begin{document}

 \maketitle

\begin{abstract}
In this paper, we characterize disjoint hypercyclic sequences of elementary operators. Also, we give some sufficient conditions for a sequence of the dual elementary operators to be disjoint topologically transitive. Finally, we give concrete examples and applications. 
\end{abstract}

\baselineskip17pt
 
\section{Introduction and Preliminaries}
 In last decades, linear dynamical properties of operators  have been studied in many research articles; see monographs \cite{bmbook,gpbook}, and recent papers \cite{iv2,tab}.  J. B\`es and A. Peris in 2007 \cite{bp07} introduced and studied a new version of disjointness for hypercyclicity, topological transitivity and topological mixing properties regarding a finite sequence of operators on a Fr\'echet space. Especially, they gave in \cite{bp07} some equivalent conditions for powers of weighted shifts on $c_0(\mathbb{Z})$ or $\ell_p(\mathbb{Z})$ ($1\leq p<\infty$) to be densely disjoint hypercyclic.  So far, this topic has been studied intensely in several papers; see for example \cite{sa11,san,bg07,bmp11,bmps12, bms14,bm12}.
 Inspired by the mentioned characterization, disjoint hypercyclicity of weighted translation operators on Lebesgue spaces in the context of locally compact groups was characterized in \cite{chen172,chen181,hl16,zhang17}. Also, the notion of disjoint topological transitivity on weighted Orlicz spaces was studied in \cite{cst,cd18}. In  \cite{saw}, the authors consider disjoint dynamics of weighted translations operators on a wide class of general Banach lattices, and give a characterization for weighted translation operators to be disjoint hypercyclic and topologically mixing on such Banach spaces. Recently, in \cite{iv2} we have characterized hypercyclic wedge operators (which are a special kind of elementary operators studied in papers \cite{bla,fia1,fia2,fia3}) on the space of compact linear operators on a separable Hilbert space. The dynamics of some similar operators have been considered earlier such as conjugate operators, see \cite{pet}, and left multiplication operators, see \cite{ZZD}, \cite{ZLD}, \cite{YRH}.    
  
  Here, we will give some necessary and sufficient conditions for a wedge operator corresponding to unitary and invertible operators  to be disjoint hypercyclic. Also, we give some sufficient conditions for a sequence of the corresponding dual operators to be disjoint topologically transitive and we give some concrete examples and applications to left multiplies. For simplicity, we work only with the notions of disjoint hypercyclicity and disjoint topological transitivity here (for some other disjoint dynamical properties of linear operators, we refer the reader to the recently published monograph \cite{nova}.

If $\mathcal X$ is a Banach space, the set of all bounded linear operators from $\mathcal X$ into $\mathcal X$ is denoted by $B(\mathcal X)$. Also, we denote $\Bbb{N}_0:=\Bbb{N}\cup\{0\}$. For convenience of readers, we recall the following definitions from \cite{bp07}.
\begin{definition}\label{defd-hyper}
	Let
	$N\geq 2$, and 	$T_{1},T_{2},\ldots,T_{N}$ 
	be bounded linear operators
	acting on a Banach space
	${\mathcal X}$. 
	\begin{enumerate}
		\item The finite sequence $T_{1},T_{2},\ldots,T_{N}$ is called \emph{disjoint hypercyclic} or simply \emph{d-hypercyclic} 
	if there exists an element $x\in \mathcal{X}$
	such that the set 
	\begin{equation}\label{15}
	\{(x,x,\ldots,x),(T_{1}x,T_{2}x,\ldots,T_{N}x),(T_{1}^{2}x,T_{2}^{2}x,\ldots,T_{N}^{2}x),\ldots\}
	\end{equation}
	is dense in $\mathcal{X}^{N}$. In this case, the element $x$ is called 
	a
	\emph{d-hypercyclic vector}.
	\item The finite sequence $T_{1},T_{2},\ldots,T_{N}$ is called \emph{disjoint topologically transitive} or simply \emph{d-topologically transitive} 
	if for any non-empty open subsets $U,V_1,\ldots,V_N$ of $\mathcal X$, 
	there exist a natural number $n\in\mathbb N$ such that
	\begin{equation}\label{00}
	U\cap T_1^{-n}(V_1)\cap\cdots\cap T_N^{-n}(V_N)\neq\varnothing.
	\end{equation}
	\end{enumerate}
\end{definition}
\section{Main Results}
In this section, we assume that $\mathcal H$ is a separable Hilbert space with an orthonormal basis $\{e_j\}_{j\in\mathbb{N}}$. Also, for each $k\in\mathbb{N}$ we denote 
$L_k:=\text{span}\{e_1,\ldots,e_k\}$, and we assume that $P_k$ is the orthogonal projection onto $L_k$.

The set of all bounded linear operators from $\mathcal H$ to $\mathcal H$ is denoted by $B(\mathcal H)$. Also, the set of all compact (finite rank, respectively) elements of $B(\mathcal H)$ is denoted by $B_0(\mathcal H)$ ($B_{00}(\mathcal H)$, respectively).
\begin{definition}
	Let $U,W\in B(\mathcal H)$. We define the operator $T_{U,W}:B(\mathcal H)\rightarrow B(\mathcal H)$ by
	\begin{equation}\label{elem}
	T_{U,W}(F):=WFU
	\end{equation}
	for all $F\in B(\mathcal H)$.
\end{definition}
Trivially, $T_{U,W}(B_0(\mathcal{H}))\subseteq B_0(\mathcal{H})$. 
If $U$ and $W$ are invertible, then $T_{U,W}$ is invertible and we have
\begin{equation*}
T^{-1}_{U,W}(F)=W^{-1}FU^{-1},\quad(F\in B(\mathcal H)).
\end{equation*}
In this case, simply we put $S_{U,W}:=T^{-1}_{U,W}$.

In next result, we give an equivalent condition for a sequence of elementary operators such as \eqref{elem} to be disjoint hypercyclic. The condition \eqref{eq000} in the hypothesis of this result plays a key role in its proof. In \cite[Example 1]{iv2} we introduce a class of unitary operators satisfying this property.

\begin{theorem} \label{thm32} 
	Let $W_1,\ldots,W_N$  be invertible bounded linear operators on $\mathcal H$. Let $U\in B(\mathcal H)$ be a unitary operator in $B(\mathcal H)$ such that for each $k\in\mathbb{N}$ there exists an $N_k\in\mathbb{N}$ with
	\begin{equation}\label{eq000}
	U^n(L_k)\perp L_k \quad\text{for all } n\geq N_k.
	\end{equation}
	For each $k\in\mathbb{N}$ denote the operator $T_{U,W_k}$ on $B_0(\mathcal H)$ by $T_k$. Also, assume that $\{r_k\}_{k=1}^N\subseteq\mathbb{N}$ such that $0<r_1<r_2<\ldots<r_N$. 
	Then, the following conditions are equivalent. 
	\begin{itemize}
		\item [{\rm (i)}] The set of all d-hypercyclic vectors of $T^{r_1}_{1},\ldots,T^{r_N}_N$ is dense in $B_0(\mathcal{H})$.
		\item [{\rm (ii)}] For each $m\in\mathbb{N}$ there exist sequences $\{D_k\}_{k=1}^\infty, \{G_k^{(1)}\}_{k=1}^\infty,\ldots, \{G_k^{(N)}\}_{k=1}^\infty$ of operators in $B_0(\mathcal{H})$, and a strictly increasing sequence $\{n_k\}_{k=1}^\infty\subseteq \mathbb{N}$ such that for each $l\in\{1,\ldots,N\}$,
		\begin{equation}\label{cond2}
		\lim_{k\rightarrow\infty}\bigl\|D_k-P_m\bigr\|=\lim_{k\rightarrow\infty}\bigl\|G^{(l)}_k-P_m\bigr\|=0,
		\end{equation}
		\begin{equation}\label{cond22}
		\lim_{k\rightarrow\infty}\bigl\|W^{r_kn_k}_lD_k \bigr\|=\lim_{k\rightarrow\infty}\bigl\|W^{-r_kn_k}_lG^{(l)}_k \bigr\|=0,
		\end{equation}
		and, for each pair of distinct $s,\ l\in\{1,\ldots,N\}$, 
		\begin{equation}\label{cond222}
		\lim_{k\longrightarrow \infty}\bigl\|W^{r_ln_k}_lW^{-r_sn_k}_sG^{(s)}_k\bigr\|=0.
		\end{equation}
	\end{itemize}
\end{theorem}
\begin{proof}
	$\text{(i)}\Rightarrow\text{(ii)}$: Let the condition ${\rm (i)}$ hold and let $m\in\mathbb{N}$. Then, by Definition \ref{defd-hyper}, there exists an integer $n_1\geq N_m$ and an operator $F_1\in B_0(\mathcal H)$ such that $\|F_1-P_m\|<\frac{1}{4}$ and $\|T_j^{r_jn_1}(F_1)-P_m\|<\frac{1}{4}$ for all $j\in\{1,\ldots,N\}$. Similarly, we can choose an integer $n_2>n_1$ and an operator $F_2\in B_0(\mathcal H)$ such that $\|F_2-P_m\|<\frac{1}{4^2}$ and $\|T_j^{r_jn_2}(F_2)-P_m\|<\frac{1}{4^2}$ for all $j\in\{1,\ldots,N\}$. Inductively, there exist a strictly increasing sequence $\{n_k\}_{k=1}^\infty\subseteq \mathbb{N}$ and a sequence $\{F_k\}_{k=1}^\infty\subseteq B_0(\mathcal H)$ such that $\|F_k-P_m\|<\frac{1}{4^k}$ and $\|T^{r_jn_k}_j(F_k)-P_m\|<\frac{1}{4^k}$ for all $k\in\mathbb{N}$ and $j\in\{1,\ldots,N\}$. Moreover, $n_k\geq N_m$ for all $k$. For each $k\in\mathbb{N}$ and $l\in\{1,\ldots,N\}$ put
	$$D_k:=F_kP_m,\qquad G^{(l)}_k:=W^{r_ln_k}_l F_k U^{r_ln_k} P_m.$$
	Then, we have
	\begin{equation}\label{cond2*}
	\lim_{k\rightarrow\infty}\bigl\|D_k-P_m\bigr\|=\lim_{k\rightarrow\infty}\bigl\|W^{r_kn_k}_lD_k \bigr\|=0
	\end{equation}
	for all $l\in\{1,\ldots,N\}$. Indeed, 
	\begin{align*}
	\frac{1}{4^k}&\geq \|F_k-P_m\|\geq \|F_k-P_m\|\,\|P_m\|\geq\|D_k-P_m\|
	\end{align*}
	and
	\begin{align*}
	\frac{1}{4^k}&\geq \|W^{r_ln_k}F_k U^{r_ln_k}-P_m\|\geq \|W^{r_ln_k}F_k U^{r_ln_k}-P_m\|\, \|U^{-r_ln_k}P_m\|\\
	&\geq \|(W^{r_ln_k}F_k U^{r_ln_k}-P_m)U^{-r_ln_k}P_m\|=\|W^{r_ln_k}D_k\|,
	\end{align*}
	for all $l\in\{1,\ldots,N\}$ and $k\in\mathbb{N}$. Similarly,
	\begin{equation}\label{cond22*}
	\lim_{k\rightarrow\infty}\bigl\|G^{(l)}_k-P_m\bigr\|=\lim_{k\rightarrow\infty}\bigl\|W^{-r_kn_k}_lG^{(l)}_k \bigr\|=0
	\end{equation}
	for all $l\in\{1,\ldots,N\}$.

	To prove the last assertion in (ii), pick distinct $s,\ l\in  \{1,\ldots,N\}$. Then, $r_l-r_s$ is a nonzero integer. Hence, by the relation \eqref{eq000} we have
	$$P_mU^{(r_l-r_s)n_k}P_m=0\qquad
	{\rm and}\qquad P_mU^{(r_s-r_l)n_k}P_m=0$$
	for all $k\in\mathbb{N}$. We get 
	\begin{align*}
	\frac{1}{4^k}&\geq \bigl\|T_l^{r_ln_k}(F_k)-P_m\bigr\|\\
	&=\bigl\|W_l^{r_ln_k}F_kU^{r_ln_k}-P_m\bigr\|\\
	&\geq \bigl\|W_l^{r_ln_k}F_kU^{r_ln_k}-P_m\bigr\| \,\bigl\|U^{(r_s-r_l)n_k}P_m\bigr\|\\
	&\geq\bigl\|W_l^{r_ln_k}F_kU^{r_sn_k}P_m-P_m U^{(r_s-r_l)n_k}P_m\bigr\|\\
	&=\bigl\|W_l^{r_ln_k}F_kU^{r_sn_k}P_m\bigr\|.
	\end{align*}
	Furthermore, we have
	\begin{align*}
	\bigl\|W_l^{r_ln_k}F_kU^{r_sn_k}P_m\bigr\|&=\bigl\|W_l^{r_ln_k}W_s^{-r_sn_k}(W_s^{r_sn_k}F_kU^{r_sn_k}P_m)\bigr\|\\
	&=\bigl\|W_l^{r_ln_k}W_s^{-r_sn_k}G_k^{(s)}\bigr\|,
	\end{align*}
	by the construction of $G_k^{(s)}$. Since $k\in\mathbb{N}$ is arbitrary, the proof of $\text{(i)}\Rightarrow\text{(ii)}$ is complete.
	
	$\text{(ii)}\Rightarrow\text{(i)}$: The main idea in this proof comes from the proof of the part $\text{(ii)}\Rightarrow\text{(i)}$ in  \cite[Theorem 2.12]{saw}. Assume that (ii) holds. By \cite[Proposition 2.3]{bp07} it would be sufficient to show that the sequence $T_1^{r_1},\ldots,T_N^{r_N}$ is $d$-topologically transitive. For this, let $O,V_1,\ldots,V_N$ be nonempty open subsets of $B_0(\mathcal H)$. Since $B_{00}(\mathcal{H})$ is dense in $B_0(\mathcal{H})$, there is a sufficiently large number $m\in\mathbb{N}$ such that $P_mF\in O$ and $P_mE_l\in V_l$ ($l=1,\ldots,N$) for some $F,E_1,\ldots,E_N\in B_0(\mathcal H)$. From the relations in (ii), it is not hard to see that 
	\begin{equation}\label{eq777}
	\lim_{k\longrightarrow \infty}\bigl\|T_l^{r_ln_k}(D_kF)\bigr\|=\lim_{k\longrightarrow \infty} \bigl\|S_l^{r_ln_k}(G_k^{(l)}E_l)\bigr\|=0
	\end{equation}
	for all $l\in\{1,\ldots,N\}$, and 
	\begin{equation}\label{eq888}
	\lim_{k\longrightarrow \infty}\Bigl\|T_l^{r_ln_k}\left(S_s^{r_sn_k}(G_k^{(s)}E_s)\right)\Bigr\|=0
	\end{equation}
	for all distinct $s,\ l\in\{1,\ldots,N\}$. Finally, put 
	$$
	\phi_k:=D_kF+\sum_{l=1}^NS_l^{r_ln_k}\bigl(G_k^{(l)}E_l\bigr).
	$$
	
	Then, for each $k\in\mathbb{N}$ we have 
	\begin{align*}
	\|\phi_k-P_mF\|&\leq \|D_kF-P_mF\|+\sum_{l=1}^N\Bigl\|S_l^{r_ln_k}(G_k^{(l)}E_l)\Bigr\|\\
	&\leq \|D_k-P_m\|\,\|F\|+\sum_{l=1}^N\Bigl\|S_l^{r_ln_k}(G_k^{(l)}E_l)\Bigr\|.
	\end{align*}
	
	This implies that $\lim\limits_{k\rightarrow\infty}\phi_k=P_mF$ in $B_0(\mathcal{H})$. Moreover, for each $1\leq l\leq N$,
	\begin{align*}
	\bigl\|T_l^{r_ln_k}(\phi_k)-P_mE_l\bigr\|&\leq\bigl\|T_l^{r_ln_k}(D_kF)\bigr\|+\bigl\|G_k^{(l)}E_l-P_mE_l\bigr\|\\
	&\hspace{3cm}+\sum_{s\neq l,\,s=1}^N \bigl\|T_l^{r_ln_k}\left(S_s^{r_sn_k}(G_k^{(s)}E_s)\right)\bigr\|\\
	&\leq\bigl\|T_l^{r_ln_k}(D_kF)\bigr\|+\bigl\|G_k^{(l)}-P_m\bigr\|\,\|E_l\|\\
	&\hspace{3cm}+\sum_{s\neq l,\,s=1}^N \bigl\|T_l^{r_ln_k}\left(S_s^{r_sn_k}(G_k^{(s)}E_s)\right)\bigr\|.
	\end{align*}
	
	Hence, by the relations \eqref{eq777} and \eqref{eq888}, for each $1\leq l\leq N$ we have 
	$\lim\limits_{k\rightarrow\infty}T_l^{r_ln_k}\phi_k=P_mE_l$ in $B_0(\mathcal{H})$.
	This completes the proof. 
\end{proof}


\begin{remark}
	One can prove a similar result for the operator $F\mapsto UFW$, where $W$ is invertible and $U$ is a unitary operator satisfying the condition \eqref{eq000}. For this, it would be enough to put $D_k:=P_mF_k$, $G_k^{(l)}:=P_mU^{r_ln_k}F_kW^{r_ln_k}_l$ and replace the relations \eqref{cond22} and \eqref{cond222} by the conditions
	\begin{equation*}
	\lim_{k\rightarrow\infty}\bigl\|D_k W^{r_kn_k}_l \bigr\|=\lim_{k\rightarrow\infty}\bigl\|G^{(l)}_k W^{-r_kn_k}_l \bigr\|=0
	\end{equation*}
	and
	\begin{equation*}
	\lim_{k\longrightarrow \infty}\bigl\|G^{(s)}_k W^{-r_sn_k}_s W^{r_ln_k}_l\bigr\|=0,
	\end{equation*}
	respectively. It follows by passing to the adjoints that if $W_{1}, \dots ,W_{N}$ satisfy the conditions \eqref{cond2}, \eqref{cond22} and \eqref{cond222}, then  $W_{1}^{*}, \dots ,W_{N}^{*}$   satisfy these new conditions.
\end{remark}

\begin{corollary}  \label{co01}
	Let $\tilde U $ be a unitary operator on $H$ satisfying the condition \eqref{eq000}. Then $T_{\tilde{U},W_{1}}^{r_{1}}, \dots , T_{\tilde{U},W_{N}}^{r_{N}}  $ are $d-$hypercyclic on $B_{0}(H) $ if and only if $T_{U,W_{1}}^{r_{1}}, \dots , T_{U,W_{N}}^{r_{N}} $ are $d-$hypercyclic for every unitary operator  $U $ on $ H.$ In other words, $T_{U,W_{1}}^{r_{1}}, \dots , T_{U,W_{N}}^{r_{N}} $ are $d-$hypercyclic on $B_{0}(H) $  for every unitary operator $U $ if and only if the operators $W_{1}, \dots, W_{N} $ satisfy the conditions \eqref{cond2}, \eqref{cond22} and \eqref{cond222}.
\end{corollary}

\begin{proof}
	If $T_{\tilde{U},W_{1}}^{r_{1}}, \dots , T_{\tilde{U},W_{N}}^{r_{N}} $ are $d-$hypercyclic, then by Theorem \ref{thm32} the conditions \eqref{cond2}, \eqref{cond22} and \eqref{cond222} are satisfied. However, then by the proof of the part $(ii) \Rightarrow (i) ,$ $T_{U,W_{1}}^{r_{1}}, \dots , T_{U,W_{N}}^{r_{N}} $ are disjoint hypercyclic for every unitary operator $U $ on $H $ since this part of the proof does not apply the condition \eqref{eq000} for unitary operators on $H.$
\end{proof}
\begin{lemma} \label{lem217}
	If the operators $W_{1}^{r_{1}} , \dots ,  W_{N}^{r_{N}} $ satisfy the conditions \eqref{cond2}, \eqref{cond22} and \eqref{cond222}, then they satisfy $d-$hypercyclicity criterion.
\end{lemma}

\begin{proof}
	Let $t \in \mathbb{N} $ be given. If $O, V_{1}, \dots , V_{t} $ are non empty open subsets of $ \underbrace{ H \oplus \dots \oplus H }_{t} =H^{t} ,$ then there exists a sufficiently large number $m \in \mathbb{N} $ such that 
	$$ (P_{m}x_{1} , \dots , P_{m}x_{t}) \in O,$$ 
	$$ (P_{m}y_{1}^{(1)} , \dots , P_{m}y_{t}^{(1)} )\in V_{1}, $$
	$$\vdots $$
	$$ (P_{m}y_{1}^{(N)} , \dots , P_{m}y_{t}^{(N)} )\in V_{N}, $$
	for some $ \lbrace x_{j} \rbrace_{1 \leq j \leq t}, \lbrace y_{j}^{(1)} \rbrace_{1 \leq j \leq t} , \dots , \lbrace y_{j}^{(N)} \rbrace_{1 \leq j \leq t} $ in $H^{t}.$ For each $k \in \mathbb{N} ,$ let \\
	$ \lbrace h_{j}^{(k)} \rbrace_{1 \leq j \leq t} \in H^{t}$ be given as 
	$$ h_{j}^{(k)} = D_{k}x_{j}+ \sum_{l=1}^{N} S_{l}^{r_{l}n_{k}} G_{k}^{(l)} y_{j}^{(l)} $$ 
	for every $j \in \lbrace 1 , \dots , t \rbrace .$ Then we can proceed as in the proof of Theorem \ref{thm32} part $ii) \Rightarrow i) $ to deduce that the operators $  \underbrace{ W_{1}^{r_{1}} \oplus \dots \oplus W_{1}^{r_{1}} }_{t} , \dots ,  \underbrace{ W_{N}^{r_{N}} \oplus \dots \oplus W_{N}^{r_{N}} }_{t} $ are $d-$topologically transitive on $H^{t} .$ Since $t \in \mathbb{N}$ was chosen arbitrary,  by \cite[Theorem 2.7]{bp07} it follows that the operators $W_{1}^{r_{1}} , \dots , W_{N}^{r_{N}} $ satisfy $d-$hypercyclicity criterion. 
\end{proof}

\begin{remark} \label{rem 2.0}
	In fact, by applying the same argument as in the proof of Lemma  \ref{lem217} we can extend proof of the part $ii) \Rightarrow i) $ in Theorem \ref{thm32} to deduce that if the operators $W_{1}^{r_{1}} , \dots , W_{N}^{r_{N}}  $ satisfy the conditions \eqref{cond2}, \eqref{cond22} and \eqref{cond222}, then the operators $ \underbrace{ T_{U,W_{1}}^{r_{1}} \oplus \dots \oplus T_{U,W_{1}}^{r_{1}} }_{t}, \dots ,  \underbrace{ T_{U,W_{N}}^{r_{N}} \oplus \dots \oplus T_{U,W_{N}}^{r_{N}}}_{t} $ are $d-$topologically transitive on $ (B_{0}(H))^{t}$ for every unitary operator $U $ on $H.$ Since this holds for arbitrary $t \in \mathbb{N} ,$ by \cite[Theorem 2.7]{bp07} we deduce the following corollary.
\end{remark}

\begin{corollary}
	The operators $W_{1}^{r_{1}} ,  \dots ,  W_{N}^{r_{N}} $ satisfy the conditions \eqref{cond2}, \eqref{cond22} and \eqref{cond222} if and only if the operators $T_{U,W_{1}}^{r_{1}} , \dots , T_{U,W_{N}}^{r_{N}} $ satisfy $d-$hypercyclicity criterion for every unitary operator $U $ on $H.$
\end{corollary}
\begin{corollary} \label{con C0103} 
	If for each distinct $ s,\ l \in \lbrace 1, \dots , N \rbrace $ we have $$\lim_{k\rightarrow \infty}  \parallel W_{l}^{r_{l}n_{k}} P_{m} \parallel = \lim_{k\rightarrow \infty} \parallel W_{s}^{-r_{s}n_{k}}P_{m}\parallel = \lim_{k\rightarrow \infty} \parallel W_{l}^{r_{l}n_{k}} W_{s}^{-r_{s}n_{k}} P_{m} \parallel =0  ,$$ 
	then $T_{U,W_{1}}^{r_{1}}, \dots , T_{U,W_{N}}^{r_{N}} $ satisfy $d-$hypercyclicity criterion for every unitary operator $U$. 
\end{corollary}	

\begin{proof}
	Put $D_{k}=G_{k}^{(l)}=P_{m} $ for all $k$ and apply Theorem \ref{thm32}.
\end{proof}

\begin{example} \label{examp 024} 
	If $H$ is separable Hilbert space, we may enumerate the orthonormal basis for $H$ by the indices in $\mathbb{Z} .$ Let then $ \lbrace e_{j} \rbrace_{j \in \mathbb{Z}} $ be an orthonormal basis for $H.$ Set, for each $m,$ $L_{m}:=Span \lbrace e_{-m},e_{-m+1}, \dots ,e_{m-1},e_{m} \rbrace.$ Then, by \cite{iv2}, there is a bijective correspondence between the set of all aperiodic maps on $\mathbb{Z} $ (where $\mathbb{Z}  $ is equipped with the discrete topology) and the set of all unitary operators on $H$ satisfying the condition \eqref{eq000}. In particular, every translation on $\mathbb{Z} $ gives a rise to a unitary operator on $H$ satisfying the condition \eqref{eq000}.\\
		
	Next, let $r_{1} \in \mathbb{N} $ and $r_{2}=2r_{1} .$ Put $W_{1} $ and $W_{2} $ to be the operators on $H$ defined as 
	$$
	W_{1}(e_{j})=
	\begin{cases}
	2e_{j+1}  & \text{  for } j < 0 ,\ \\
	\dfrac{1}{2}e_{j+1}  & \text{  for } j \geq 0 ; \ \\
	\end{cases} 
	$$
	
	$$
	W_{2}(e_{j})=
	\begin{cases}
	3e_{j+1}  & \text{  for } j < 0 ,\ \\
	\dfrac{1}{3}e_{j+1}  & \text{  for } j \geq 0 . \ \\
	\end{cases} 
	$$
	Then $W_{1} $ and $W_{2} $ are bounded operators. By some calculations, it is not hard to see that for each $m \in \mathbb{N}$ we have 
	
	$$ \parallel W_{1}^{r_{1}n_{k}} W_{2}^{-r_{2}n_{k}} P_{m} \parallel = 
	\parallel W_{1}^{r_{1}n_{k}} W_{2}^{-2r_{1}n_{k}} P_{m} \parallel \leq   $$

	$$
	\leq \dfrac{ 3^{2m}2^{r_{1}n_{k}}}{3^{2r_{1}n_{k}}} < \dfrac {3^{2m}} {3^{r_{1}n_{k}}}  \rightarrow 0  \text{  as } k \rightarrow \infty.$$
	
	Similarly, 
	$$\parallel W_{2}^{r_{2}n_{k}} W_{1}^{-r_{1}n_{k}} P_{m} \parallel = 
	\parallel W_{2}^{2r_{1}n_{k}} W_{1}^{-r_{1}n_{k}} P_{m} \parallel < \dfrac {3^{2m}} {2^{r_{1}n_{k}}}  \rightarrow 0  \text{  as } k \rightarrow \infty $$ 
	for each $m \in \mathbb{N} .$ Moreover, by the similar arguments as in \cite{iv2}, we can deduce that for each $ m \in \mathbb{N}  $ we have 
	$$\lim_{k\rightarrow \infty}  \parallel W_{1}^{r_{1}n_{k}} P_{m} \parallel = \lim_{k\rightarrow \infty} W_{1}^{-r_{1}n_{k}}P_{m} \parallel = \lim_{k\rightarrow \infty} \parallel W_{2}^{r_{2}n_{k}}  P_{m} \parallel \lim_{k\rightarrow \infty} \parallel W_{2}^{-r_{2}n_{k}}  P_{m} \parallel   =0  .$$ 
	By Corollary \ref{con C0103} it follows that $T_{U,W_{1}}^{r_{1}} $ and $T_{U,W_{2}}^{r_{2}}  $ satisfy $d-$hypercyclicity criterion for every unitary operator $U.$
	
	In general, consider $\ell_{2}(\mathbb{Z}) $ and let $ \lbrace z_{j} \rbrace$ be the natural orthonormal basis for $\ell_2 (\mathbb{Z})  .$ Let $ V $ be the unitary operator from $ \ell_2 (\mathbb{Z})$ onto $ H $ given by $Vz_{j}=e_{j} $ for all $j \in \mathbb{Z} .$ By the previous arguments it follows that if $\tilde{T_{1}}^{r_{1}} , \dots , \tilde{T_{N}}^{r_{N}} $ are disjoint hypercyclic weighted  shifts on $\ell_2 (\mathbb{Z}) ,$ then $T_{1}^{r_{1}} , \dots , T_{N}^{r_{N}} $ satisfy $d-$hypercyclicity criterion on $B_{0}(H) $ where $T_{i}^{r_{i}}(F)=V \tilde{T_{i}}^{r_{i}}V^{*}FU^{r_{i}} $ for all $i \in \lbrace 1, \dots , N \rbrace $ and $ F \in B_{0} (H)$ and $U$ is an arbitrary unitary operator on $H.$ For more details about disjoint hypercyclic  weighted shifts, see \cite{bms14}, \cite{bp07}.
\end{example}
\begin{proposition} \label{pro02}
	Suppose that there exist dense subsets $H_{0},H_{1}, \dots , H_{N} $ of $H$ and a strictly increasing sequence $\lbrace n_{k} \rbrace_{k} $ such that $ W_{l}^{r_{l}n_{k}} \stackrel{k \rightarrow \infty }{\longrightarrow} 0 $ pointwise on $H_{0} $ for every $l \in \lbrace 1, \dots , N \rbrace $ and $ W_{l}^{-r_{l}n_{k}} \stackrel{k \rightarrow \infty }{\longrightarrow} 0$ pointwise on $H_{l}$ for every $l \in \lbrace 1, \dots , N \rbrace.$ Suppose in addition that $  W_{l}^{r_{l}n_{k}} W_{s}^{-r_{s}n_{k}}  \stackrel{k \rightarrow \infty }{\longrightarrow} 0 $ pointwise on $H_{s} $ whenever $s \neq l .$ Then $T_{U,W_{1}}^{r_{1}}, \dots T_{U,W_{N}}^{r_{N}}  $ satisfy $d-$hypercyclicity criterion for every unitary operator $ U.$
\end{proposition}

\begin{proof}
	The main idea in this proof is motivated by the proof of \cite[Proposition 3.3]{ZZD}. Given $m \in N ,$ for every $j \in \lbrace -m, \dots , m \rbrace $ we can find sequences $$\lbrace {}^{(j)} f_{i} \rbrace_{i} , \lbrace {}^{(j)} g_{i}^{(1)} \rbrace_{i} , \dots ,  \lbrace {}^{(j)}g_{i}^{(N)} \rbrace_{i}  $$ 
	such that $\lbrace {}^{(j)} f_{i} \rbrace_{i}  \subseteq  H_{0},\lbrace {}^{(j)} g_{i}^{(l)} \rbrace_{i} \subseteq H_{l}  $ for every $ l\in \lbrace 1, \dots N \rbrace ,$ $ {}^{(j)} f_{i} \rightarrow e_{j} $   and $ {}^{(j)} g_{i}^{(l)} \rightarrow e_{j} $ as $i \rightarrow \infty $ for every $j \in \lbrace -m, \dots , m \rbrace  $ and $l \in \lbrace 1, \dots , N \rbrace  .$ By the assumption, we can find a subsequence $ \lbrace n_{k_{i}} \rbrace_{i}$ of $ \lbrace n_{k} \rbrace_{k}$ such that $$\parallel W_{l}^{r_{l}n_{k_{i}}} \text{ }{}^{(j)}f_{i} \parallel < \frac{1}{2mi} , \text{ } 
	\parallel W_{l}^{-r_{l}n_{k_{i}}} \text{ }{}^{(j)}g_{i}^{(l)} \parallel < \frac{1}{2mi} $$ 
	and
	$$\parallel W_{l}^{r_{l}n_{k_{i}}}  W_{s}^{-r_{s} n_{k_{i}}}  \text{ }{}^{(j)}g_{i}^{(s)} \parallel < \frac{1}{2mi}   $$ 
	for all $i \in \mathbb{N} ,$ $j \in \lbrace -m, \dots , m \rbrace $   and $s,l \in \lbrace 1, \dots , N \rbrace $ with $s \neq l .$ Indeed, from the assumptions in the proposition, it is not hard to see that we can construct such subsequence $\lbrace n_{k_{i}} \rbrace_{i \in \mathbb{N}} .$ For each $i \in \mathbb{N} ,$ put $D_{i}, G_{i}^{(l)} $ to be the operators on $H$ given by 
	$$D_{i} e_{j}=\begin{cases}
	{}^{(j)}f_{i}  & \text{  for } j \in \lbrace -m, \dots , m \rbrace \,\\
	0  & \text{  else, }  \,\\
	\end{cases}  $$ 
	
	$$G_{i}^{(l)} e_{j}=\begin{cases}
	{}^{(j)}g_{i}^{(l)}  & \text{  for } j \in \lbrace -m, \dots , m \rbrace \,\\
	0  & \text{  else, }  \,\\
	\end{cases}  $$
	where $ l \in \lbrace 1, \dots , N \rbrace .$ It is then not hard to see that 
	$$ \lim_{i \rightarrow \infty} \parallel D_{i}-P_{m} \parallel =  \lim_{i \rightarrow \infty} \parallel G_{i}^{(l)} - P_{m} \parallel = 0$$ 
	and 
	$$ \lim_{i \rightarrow \infty} \parallel W_{l}^{r_{l} n_{k_{i}}} D_{i} \parallel = \lim_{i \rightarrow \infty} \parallel W_{l}^{-r_{l} n_{k_{i}}} G_{i}^{(l)} \parallel = 0. $$ 
	Moreover, for all $s,l \in \lbrace 1, \dots , N \rbrace $ with $s \neq l $ we have
	$$	\lim_{i\longrightarrow \infty}\bigl\|W^{r_l n_{k_{i}}}_lW^{-r_s n_{k_{i}}}_sG^{(s)}_i\bigr\|=0.$$
	Here we use that strong and uniform convergence coincide for operators on a finite dimensional subspace. By part $(ii) \Rightarrow (i) $ in Theorem \ref{thm32} and Remark \ref{rem 2.0} the result follows.
\end{proof}


\begin{example} \label{ex03}
	Let $H=L^{2}(\mathbb{R}), \text{ } \alpha(t)=t-1 $ for all $t \in \mathbb{R},$ $w_{1}=2 {\mathcal X}_{\mathbb{R^{-}}} + \frac{1}{2} {\mathcal X}_{\mathbb{R^{+}}} $ and $w_{1}=3 {\mathcal X}_{\mathbb{R^{-}}} + \frac{1}{3} {\mathcal X}_{\mathbb{R^{+}}} .$ Choose an $ r_{1} \in \mathbb{N} $ and set $r_{2}=2r_{1} .$ \\
	Let $W_{1}, W_{2} \in B(H) $ be given by $W_{j}(f)=w_{j} \cdot (f \circ \alpha_{j}) $ for all $f \in H $ and $j \in \lbrace 1,2 \rbrace .$ \\
	For every $m \in \mathbb{N} $ and $n > m $ we have 
	
	$$ \sup_{t\in [-m,m]}   \dfrac{\vert \prod_{j=1}^{r_{2}n} (w_{2} \circ \alpha^{r_{1} n - j})(t) \vert}    {\vert  \prod_{j=0}^{r_{1}n-1} (w_{1} \circ \alpha^{j})(t)  \vert} $$
	
	$$= \sup_{t\in [-m,m]} \vert  \text{ } (\prod_{j=0}^{r_{1}n} (w_{2} \circ \alpha^{-j}) ) (\prod_{j=0}^{r_{1}n-1} (\frac{w_{2}}{w_{1}}  \circ \alpha^{j} )) \text{ } \vert  $$
	
	$$\leq \dfrac{3^{m+1}}{3^{r_{1}n-m}} {( \frac{3}{2})}^{r_{1}n} = \dfrac{3^{2m+1}}{2^{r_{1}n}} \stackrel{n \rightarrow \infty }{\longrightarrow} 0 , $$
	
	$$ \sup_{t\in [-m,m]}   \dfrac{ \vert \prod_{j=1}^{r_{1}n} (w_{1} \circ \alpha^{r_{2} n - j})(t) \vert}    { \vert   \prod_{j=0}^{r_{2}n-1} (w_{2} \circ \alpha^{j})(t) \vert} =$$
	
	$$= \sup_{t\in [-m,m]} \vert  \text{ } (\prod_{j=0}^{r_{1}n} (w_{2} \circ \alpha^{j})^{-1} \text{ }) (  \prod_{j=1}^{r_{1}n} (\frac{w_{1}}{w_{2}}  \circ \alpha^{r_{2}n-j} )) \text{ } \vert  $$
	
	$$\leq \dfrac{3^{m+1}}{3^{r_{1}n-m}} {( \frac{2}{3})}^{r_{1}n} \leq \dfrac{3^{2m+1}}{2^{r_{1}n}} \stackrel{n \rightarrow \infty }{\longrightarrow} 0 , $$
	
	
	$$\sup_{t\in [-m,m]} \vert  \text{ } \prod_{j=0}^{r_{1}n-1} (w_{1} \circ \alpha^{j -r_{1}n}) (t) \text{ } \vert \leq \dfrac{1}{2^{r_{1}n-2m}} \stackrel{n \rightarrow \infty }{\longrightarrow} 0 ,$$ 
	
	$$\sup_{t\in [-m,m]} \vert  \text{ } \prod_{j=0}^{r_{1}n-1} (w_{1} \circ \alpha^{j })^{-1} (t) \text{ } \vert \leq \dfrac{1}{2^{r_{1}n-2m}} \stackrel{n \rightarrow \infty }{\longrightarrow} 0 .$$ 
	
	Likewise we have that
	$$\sup_{t\in [-m,m]} \vert  \text{ } \prod_{j=0}^{r_{2}n-1} (w_{2} \circ \alpha^{j -r_{2}n}) (t) \text{ } \vert  \stackrel{n \rightarrow \infty }{\longrightarrow} 0 ,$$ 
	and
	$$\sup_{t\in [-m,m]} \vert  \text{ } \prod_{j=0}^{r_{2}n-1} (w_{2} \circ \alpha^{j})^{-1} (t) \text{ } \vert  \stackrel{n \rightarrow \infty }{\longrightarrow} 0 ,$$ 	
	
	By some calculations it is not hard to see for every $f$ with supp $f \subseteq [-m,m] $ that 
	$$ \int \vert W_{1}^{r_{1}n}(f) \vert^{2} \,d\lambda \leq \sup_{t\in [-m,m]} \vert  \text{ } \prod_{j=0}^{r_{1}n-1} (w \circ \alpha^{j-r_{1}n}) (t) \text{ } \vert^{2} \parallel f \parallel_{2}^{2} \text{ } \stackrel{n \rightarrow \infty }{\longrightarrow} 0 $$ 
	and  
	$$\int \vert W_{2}^{r_{2}n} W_{1}^{-r_{1}n}(f) \vert^{2} \,d\lambda \leq \sup_{t\in [-m,m]} 
	\dfrac{ \vert \prod_{j=1}^{r_{2}n} (w_{2} \circ \alpha^{r_{1} n - s})(t) \vert^{2}}    { \vert   \prod_{j=0}^{r_{1}n} (w_{1} \circ \alpha^{j})(t) \vert^{2}} \stackrel{n \rightarrow \infty }{\longrightarrow} 0 . $$
	
	Similarly, by the above relations we can deduce that 
	$$  W_{2}^{r_{2}n}f \stackrel{n \rightarrow \infty }{\longrightarrow} 0  \text{ and }  W_{1}^{r_{1}n}  W_{2}^{-r_{2}n}f \stackrel {n \rightarrow \infty }{\longrightarrow} 0.$$ 
	
	Moreover, 
	$$ \int \vert W_{1}^{r_{1}n}(f) \vert^{2} \,d\lambda \leq \sup_{t\in [-m,m]} \vert  \text{ } \prod_{j=0}^{r_{1}n-1} (w \circ \alpha^{j})^{-1} (t) \text{ } \vert^{2} \parallel f \parallel_{2}^{2} \text{ } \stackrel{n \rightarrow \infty }{\longrightarrow} 0  $$ 
	and, similarly, $   W_{2}^{-r_{2}n}f \stackrel{n \rightarrow \infty }{\longrightarrow} 0 .$ Since $m \in \mathbb{N} $ was chosen arbitrary, it follows that the conditions of Proposition \ref{pro02} are satisfied because $C_{c} (\mathbb{R}) $ is dense in $L^{2}(\mathbb{R}) .$ Hence $T_{U,W_{1}}^{r_{1}} $ and $T_{U,W_{2}}^{r_{2}} $ satisfy $d-$hypercyclicity criterion for every unitary operator $ U.$ 
	
	In general, if $ \alpha $ is a translation on $\mathbb{R} $ and $ w_{1}, \dots , w_{N}$ are positive, measurable, bounded weight functions satisfying that $ w_{1}^{-1}, \dots , w_{N}^{-1} $ are also bounded, then we can consider the corresponding sequence $W_{1}^{r_{1}} , \dots ,W_{N}^{r_{N}} $ of weighted translation operators on $L^{2}( \mathbb{R}).$ This sequence would satisfy the conditions of  Proposition \ref{pro02}. if for every $l,s \in \lbrace 1, \dots , N \rbrace $ and $m \in \mathbb{N} $ we have that 
	$$ \lim_{n \rightarrow \infty } \sup_{t\in [-m,m]} \vert  \text{ } \prod_{j=0}^{r_{l}n-1} (w_{l} \circ \alpha^{j-r_{l}n}) (t) \text{ } \vert  =  
	\lim_{n \rightarrow \infty } \sup_{t\in [-m,m]} \vert  \text{ } \prod_{j=0}^{r_{l}n-1} (w_{l} \circ \alpha^{j})^{-1} (t) \text{ } \vert =0 $$ 
	and in addition 
	$$\lim_{n \rightarrow \infty } \sup_{t\in [-m,m]}      \dfrac{ \vert \prod_{j=1}^{r_{l}n} (w_{l} \circ \alpha^{r_{s} n - j})(t) \vert}    { \vert   \prod_{j=0}^{r_{s}n-1} (w_{s} \circ \alpha^{j})(t) \vert} =0 .$$ 
	For more details, see \cite{saw}.
	
\end{example}


\begin{remark}
	Theorem \ref{thm32} is also valid if we replace $B_0(\mathcal{H})$ by the space of all trace class operators equipped with the trace norm which usually is denoted by $(B_1(\mathcal{H}),\|\cdot\|_1)$. In this case, we should replace the norm in the relations \eqref{cond2}-\eqref{cond222} by the trace norm. In order to prove the implication (i)$\Rightarrow$(ii), just recall that $\|FD\|_1\leq \|F\|_1\,\|D\|$ for all $F\in B_1(\mathcal H),\,D\in B(\mathcal H)$, and proceed similar to the proof of Theorem \ref{thm32}. For the proof of (ii)$\Rightarrow$(i), just observe that the set $\{P_mF:\,F\in B_1(\mathcal H),\,\, m\in\mathbb{N}\}$ is dense in $B_1(\mathcal{H})$. To see this, recall first that $B_{00}(\mathcal H)$ is dense to $B_1(\mathcal H)$. Now, if $F\in B_{00}(\mathcal H)$, then $P_{{\rm Im}(F)}$ is a finite rank operator and $P_{{\rm Im}(F)}F=F$. Hence,
	\begin{align*}
	\|P_mF-F\|_1&=\|P_mP_{{\rm Im}(F)}F-P_{{\rm Im}(F)}F\|_1\\
	&\leq \|P_mP_{{\rm Im}(F)}-P_{{\rm Im}(F)}\|\,\|F\|_1\rightarrow 0,
	\end{align*}
	as $m\rightarrow\infty$.
	
	Next, since the convergence in the trace norm coincide with the convergence in the operator norm for operators acting on a finite dimensional subspace, it follows that Corollary \ref{con C0103}, Example \ref{examp 024}, Proposition \ref{pro02} and Example \ref{ex03} are also valid in this case.
	
	The similar statement holds for the space of Hilbert-Schmidt operators equipped with the Hilbert-Schmidt norm.  
\end{remark}

\begin{proposition} \label{pro216} 
	Let $W_{1} , \dots , W_{N} $ be bounded linear operators on $H$ and $ \lbrace r_{k} \rbrace_{k=1}^{N} \subseteq \mathbb{N} $ such that $0 < r_{1} < \dots < r_{N} .$ If $W_{1}^{r_{1}} , \dots ,  W_{N}^{r_{N}}$ satisfy $d-$hypercyclicity criterion on $H,$ then 
	$T_{U, W_{1}}^{r_{1}} , \dots , T_{U, W_{N}}^{r_{N}} $ 
	satisfy $d-$hypercyclicity criterion on $ B_{0}(H) ,$  $ B_{1}(H) $ and $ B_{2}(H)$
	for every unitary operator $U $ on $ H.$ 
\end{proposition}

\begin{proof}
	If $W_{1}^{r_{1}} , \dots ,  W_{N}^{r_{N}} $ satisfy $d-$hypercyclicity criterion, then there exist dense subsets $H_{0}, \dots , H_{N} $ of $H$ and the mappings $V_{l,k} : H_{l} \rightarrow H \text{ } (1 \leqq l \leqq N, \text{ } k \in \mathbb{N})$ satisfying that $W_{l}^{r_{l}n_{k}} \stackrel{k \rightarrow \infty }{\longrightarrow} 0 $ pointwise on $H_{0} $ for every $l \in \lbrace 1 , \dots , N \rbrace ,$ $V_{l,k} \stackrel{k \rightarrow \infty }{\longrightarrow} 0  $  pointwise on $H_{l} $ for every $l \in \lbrace 1 , \dots , N \rbrace  ,$ $W_{l}^{r_{l}n_{k}}V_{s,k} \stackrel{k \rightarrow \infty }{\longrightarrow} 0$ pointwise on $H_{s} $ whenever $s \neq l, \text{ } s,l \in \lbrace 1 , \dots , N \rbrace $ and $W_{l}^{r_{l}n_{k}}V_{l,k}  \stackrel{k \rightarrow \infty }{\longrightarrow} Id_{H_{l}} $
	pointwise on $H_{l} $ for every $l \in \lbrace 1 , \dots , N \rbrace .$ By the similar arguments as in the proof of Proposition \ref{pro02}, given $m \in \mathbb{N} ,$ for $j \in \lbrace -m , \dots , m \rbrace $ we can construct sequences $\lbrace {}^{(j)}f_{i} \rbrace ,$ $\lbrace {}^{(j)}g_{i}^{(1)} \rbrace, $ $ \cdots ,\lbrace {}^{(j)}g_{i}^{(N)} \rbrace $ in $H$ and a subsequence $\lbrace n_{k_{i}} \rbrace_{i} $ of $\lbrace n_{k} \rbrace_{k} $ such that all the properties in the proof of Proposition \ref{pro02} are satisfied exept that in this case we assume that 
	
	$$\parallel V_{l,k_{i}} \text{ }{}^{(j)}g_{i}^{(l)}  \parallel < \frac{1}{2mi}   \text{ and } 
	\parallel W_{l}^{r_{l}n_{k_{i}}} \text{ } V_{s,k_{i}} \text{ }{}^{(j)}g_{i}^{(s)}  \parallel < \frac{1}{2mi} $$  
	for all $i \in \mathbb{N} ,$ and $s,l \in \lbrace 1, \dots , N \rbrace $ with $s \neq l .$ Moreover, these sequences can be constructed in a such way that in addition 
	$$\parallel W_{l}^{r_{l}n_{k_{i}}} \text{ } V_{l,k_{i}} \text{ } {}^{(j)}g_{i}^{(l)} -{}^{(j)}g_{i}^{(l)}  \parallel < \frac{1}{2mi}$$ 
	for all $i \in \mathbb{N} ,$ $l \in \lbrace 1, \dots , N \rbrace  $ and $j \in \lbrace -m, \dots , m \rbrace .$ This follows since 
	$$W_{l}^{r_{l}n_{k}} \text{ } V_{l,k}  \stackrel{k \rightarrow \infty }{\longrightarrow} Id_{H_{l}}  $$ 
	pointwise on $H_{l} $ for every $l \in \lbrace 1, \dots , N \rbrace .$ In exactly the same way as in the proof of Proposition \ref{pro02} we can then construct the finite rank operators $D_{i}, G_{i}^{(l)} .$
	
	Since ${}^{(j)}g_{i}^{(l)} \stackrel{i \rightarrow \infty }{\longrightarrow} e_{i} $ for all $ j \in \lbrace -m, \dots ,m \rbrace $ and $l \in \lbrace 1, \dots , N \rbrace  ,$ there exists an $I >0 $ such that $ \lbrace {}^{(j)}g_{i}^{(l)} \rbrace_{-m \leq j \leq m} $ is a linearly independent set for $i \geq I $ and $l \in \lbrace 1, \dots , N \rbrace .$ For each $i \geq I $ and $l \in \lbrace 1, \dots , N \rbrace ,$ set $\tilde{V}_{l,i} $ to be the operator which is $0 $ on $Span \lbrace {}^{(j)}g_{i}^{(l)} \rbrace_{-m \leq j \leq m}^{\perp} $ and on $Span \lbrace {}^{(j)}g_{i}^{(l)} \rbrace_{-m \leq j \leq m}  $ it is defined as 
	$\tilde{V}_{l,i} ({}^{(j)}g_{i}^{(l)})= V_{l,k_{i}} ({}^{(j)}g_{i}^{(l)}) $ for $ j \in \lbrace -m, \dots ,m \rbrace .$ Again, since for operators acting on a finite dimensional subspace the convergence in operator norm, trace norm and Hilbert. Schmidt norm coincide with the pointwise convergence on orthonormal basis vectors (which are finitely many in this case), it follows that $D_{i}  \stackrel{i \rightarrow \infty }{\longrightarrow} P_{m}$ and $G_{i}^{(l)} \stackrel{i \rightarrow \infty }{\longrightarrow} P_{m} $ for $l \in \lbrace 1, \dots , N \rbrace $ in operator norm, trace norm and Hilbert - Schmidt norm. Similarly, 
	$ W_{l}^{r_{l}n_{k_{i}}} D_{i}  \stackrel{i \rightarrow \infty }{\longrightarrow} 0,$ 
	$ W_{l}^{r_{l}n_{k_{i}}} \tilde{V}_{s,i} G_{i}^{(s)} \stackrel{i \rightarrow \infty }{\longrightarrow} 0,$ $\tilde{V}_{l,i} G_{i}^{(l)} \stackrel{i \rightarrow \infty }{\longrightarrow} 0$  
	and 
	$ W_{l}^{r_{l}n_{k_{i}}} \tilde{V}_{l,i} G_{i}^{(l)} \stackrel{i \rightarrow \infty }{\longrightarrow} P_{m} $
	for $s,l \in \lbrace 1, \dots , N \rbrace $ with $s \neq l $ where the convergence is in operator norm, trace norm and Hilbert - Schmidt norm. Then we can proceed further in exactly the same way as in the proof of part $ii)$ implies $i)$ in Theorem \ref{thm32} except that in this case we consider 
	$$\phi_{i} = D_{k_{i}} F + \sum_{l=1}^N  \tilde{V}_{l,i} G_{i}^{(l)} E_{l} .$$ 
	
	In fact, as explained in Remark \ref{rem 2.0} we can extend this proof to obtain that given $t \in \mathbb{N} ,$ the operators  $ \underbrace{  T_{U,W_{1}}^{r_{1}} \oplus \dots \oplus T_{U,W_{1}}^{r_{1}} }_{t} , \dots ,  \underbrace{ T_{U,W_{N}}^{r_{N}} \oplus \dots \oplus T_{U,W_{N}}^{r_{N}} }_{t} $ are $d-$topologically transitive on $(B_{0}(H))^{t} $ for every unitary operator $ U $ on $ H .$ Indeed, if $O,S_{1} , \dots , S_{N} $ are non empty open subsets of $(B_{0}(H))^{t} ,$ then there exists some $m \in \mathbb{N} $ such that 
	$$(P_{m}F_{1}, \dots , P_{m}F_{t}) \in O,  (P_{m}E_{1}^{(1)}, \dots , P_{m}E_{t}^{(1)} ) \in S_{1} , \dots , ( P_{m}E_{1}^{(N)}, \dots , P_{m}E_{t}^{(N)} ) \in S_{t}   $$
	for some sequences $\lbrace F_{j} \rbrace_{1 \leq j \leq t} , \lbrace E_{j}^{(1)} \rbrace_{1 \leq j \leq t} ,  \dots , \lbrace E_{j}^{(N)} \rbrace_{1 \leq j \leq t}   $ in $B_{0}(H) .$ Then we can consider the vectors $\phi_{i} ,$ in $(B_{0}(H))^{t} $ where 
	$$(\phi_{i})_{j}=D_{k}F_{j}+\sum_{l=1}^{N} \tilde{V}_{l.i} G_{i}^{(l)} E_{j}^{(l)} ,$$ 
	for every $ j \in \lbrace 1,  \dots , t \rbrace.$ The approach is the same in the case when $B_{0}(H)$  is replaced by $B_{1}(H) $ or $B_{2}(H) .$ Since $t \in \mathbb{N} $ was arbitrary, by \cite[Theorem 2.7]{bp07} we can deduce that the operators $ T_{U,W_{1}}^{r_{1}} , \dots , T_{U,W_{N}}^{r_{N}} $ satisfy $d-$hypercyclicity criterion on $B_{0}(H) , B_{1}(H)$ and $B_{2}(H) $ for every unitary operator $U $ on $H .$
\end{proof}

Now we are ready to present the main result in this paper which is motivated by \cite{ZZD}  and \cite{YRH}.

\begin{theorem} \label{the 2.12}
	Let $\tilde{U} $ be a unitary operator on $H$ satisfying the condition \eqref{eq000}. Then $T_{\tilde{U},W_{1}}^{r_{1}}, \dots , T_{\tilde{U},W_{N}}^{r_{1N}} $ satisfy $d-$hypercyclicity criterion on $B_{0}(H) $ if and only if $W_{1}^{r_{1}} , \dots ,  W_{N}^{r_{N}}  $ satisfy $d-$hypercyclicity criterion on $H.$ Hence $W_{1}^{r_{1}} , \dots ,  W_{N}^{r_{N}}  $ satisfy $d-$hypercyclicity criterion on $H$ if and only if $T_{U,W_{1}}^{r_{1}} , \dots , T_{U,W_{N}}^{r_{N}}  $ satisfy $d-$hypercyclicity criterion on $B_{0}(H) $ for every unitary operator $U .$ The similar statements hold if we consider $B_{1}(H) $ or $B_{2}(H) $ instead of $B_{0}(H) .$
\end{theorem}

\begin{proof}
	Combine Lemma \ref{lem217} and Proposition \ref{pro216}. 
\end{proof}	

Next we give an application of Theorem \ref{the 2.12} to left multipliers. For an operator $T$ in $B(H),$ we will denote the left an the right multiplier by $L_{T} $ and $R_{T} ,$ respectively. 

\begin{corollary}
	Let $W_{1} , \dots , W_{N} $ be invertible bounded linear operators on $H$ and $\lbrace r_{k} \rbrace_{1 \leq k \leq N} \subseteq \mathbb{N} $ such that $0 \leq r_{1} \leq \dots \leq r_{N} .$ Then $L_{W_{1}^{r_{1}}} , \dots , L_{W_{N}^{r_{N}}} $  satisfy $d-$hypercyclicity criterion on $B_{0}(H) $ if and only if $W_{1}^{r_{1}} , \dots , W_{N}^{r_{N}} $  satisfy $d-$hypercyclicity criterion on $H.$ The similar statements hold if we replace $B_{0}(H) $ by $B_{1} (H)$ or $B_{2}(H).$
\end{corollary}

\begin{proof}
	Let $ \tilde{U} $ be a unitary operator on $H$ satisfying the condition \eqref{eq000}. Then we have 
	$$ L_{W_{l}^{r_{l}n}}  R_{\tilde{U}^{r_{l}n}} =  R_{\tilde{U}^{r_{l}n}} L_{W_{l}^{r_{l}n}} = T_{\tilde{U}, W_{l}}^{r_{l}n}  $$ 
	for every $l \in \lbrace 1, \dots ,N \rbrace $ and all $n \in \mathbb{N}.$ If $ L_{W_{l}^{r_{l}}} \dots  L_{W_{N}^{r_{N}}} $  satisfy $d-$hypercyclicity criterion on $ B_{0} (H),$ then there exist dense subsets $X_{0}, X_{1}, \dots , X_{N} $ of $  B_{0} (H) ,$ the mappings $ V_{l,k} : X_{l} \rightarrow B_{0} (H) ( 1 \leq l \leq N)$ and a strictly increasing sequence $ \lbrace n_{k} \rbrace_{k} \subseteq \mathbb{N} $ such that for all $s,l \in \lbrace 1, \dots , N \rbrace $ with $ s \neq l$ we have that 
	$$L_{W_{l}^{r_{l}n_{k}}} \stackrel{k \rightarrow \infty}{\longrightarrow} 0  \text{ pointwise on } X_{0} , \text{ } V_{l,k}  \stackrel{k \rightarrow \infty}{\longrightarrow} 0 \text{ pointwise on } X_{l} ,$$
	$$L_{W_{l}^{r_{l}n_{k}}}  V_{s,k} \stackrel{k \rightarrow \infty}{\longrightarrow} 0 \text{ pointwise on } X_{s} , \text{ } L_{W_{l}^{r_{l}n_{k}}}  V_{l,k} \stackrel{k \rightarrow \infty}{\longrightarrow} Id_{X_{l}} \text{ pointwise on } X_{l}.$$ 
	
	Hence $T_{\tilde{U}, W_{l}}^{r_{l}n_{k}}  \stackrel{k \rightarrow \infty}{\longrightarrow} 0 $ pointwise on $X_{0} $ and $R_{\tilde{U}^{*}}^{r_{l}n_{k}} V_{{l},{k}}  \stackrel{k \rightarrow \infty}{\longrightarrow} 0$ pointwise on $X_{l} $ for every $l \in \lbrace 1, \dots , N \rbrace $ since $R_{\tilde{U}^{*}} $ is an isometry on $B_{0} (H) .$ 
	Moreover, 
	$$ L_{W_{l}^{r_{l}n_{k}}}  V_{l,k} = 
	L_{W_{l}^{r_{l}n_{k}}} R_{\tilde{U}}^{r_{l}n_{k}}  R_{\tilde{U}^{*}}^{r_{l}n_{k}}  V_{{l},{k}} = 
	T_{\tilde{U}, W_{l}}^{r_{l}n_{k}} R_{\tilde{U}^{*}}^{r_{l}n_{k}}  V_{l,k} $$ 
	so
	$$ T_{\tilde{U}, W_{l}}^{r_{l}n_{k}} R_{\tilde{U}^{*}}^{r_{l}n_{k}}  V_{l,k} \stackrel{n \rightarrow \infty}{\longrightarrow} Id_{X_{l}} $$	
	pointwise on $X_{l} $ for every $ l \in \lbrace 1, \dots , N \rbrace .$ Finally, again since $R_{\tilde{U}^{*}} $ is an isometry on $B_{0} (H) ,$ for all $s,l \in \lbrace 1, \dots , N \rbrace $ with $ s \neq l$ we have that 
	$$ R_{\tilde{U}^{*}}^{ (r_{l} - r_{s})n_{k}}  L_{W_{l}^{r_{l}n_{k}}}  V_{s,k} \stackrel{k \rightarrow \infty}{\longrightarrow} 0 $$ 
	pointwise on $ X_{s} .$ However, 
	$$  R_{\tilde{U}^{*}}^{ (r_{l} - r_{s})n_{k}}  L_{W_{l}^{r_{l}n_{k}}}  V_{s,k} = T_{\tilde{U} , W_{l}}^{r_{l}n_{k}} R_{\tilde{U}^{*}}^{r_{s}n_{k}}  V_{s,k}  .$$ 
	Thus, the operators $T_{\tilde{U},W_{1}}^{r_{1}} , \dots , T_{\tilde{U},W_{N}}^{r_{N}} $ satisfy $d-$hypercyclicity criterion. By Theorem \ref{thm32} it follows that $W_{1}^{r_{1}}, \dots , W_{n}^{r_{N}} $ satisfy the conditions \eqref{cond2}, \eqref{cond22} and \eqref{cond222}, hence, by Lemma  \ref{lem217} the operators $W_{1}^{r_{1}}, \dots , W_{n}^{r_{N}} $ satisfy $d-$hypercyclity criterion on $H.$ The implication in the opposite direction follows from Proposition \ref{pro216}.
\end{proof}

Let $(B(H), SOT) $ denote the space $B(H)$ equipped with the strong operator topology. Using that the set $\lbrace P_{m}F \text{ } \vert \text{ } F \in B(H) $ and $m \in \mathbb{N} \rbrace $ is dense in $(B(H), SOT)$ we can proceed in exactly the same way as in the proof of the part $(ii) \Rightarrow (i) $ of Theorem \ref{thm32} to deduce the following. 

\begin{corollary} \label{coro210}
	We have $(ii) $ implies $(i). $\\
	$(i) \text{ } T_{U,W_{1}}^{r_{1}}, \dots , T_{U,W_{N}}^{r_{N}}$ are disjoint topologically transitive in $(B(H), SOT)$ for every unitary operator $U.$ \\
	$(ii)$ For each $m \in \mathbb{N} $ there exist sequences $\lbrace D_{k} \rbrace_{k=1}^{\infty}, \lbrace G_{k}^{(1)} \rbrace_{k=1}^{\infty}, \dots , \lbrace G_{k}^{(N)} \rbrace_{k=1}^{\infty}  $ of operators in $B(H)$ and a strictly increasing sequence $\lbrace n_{k} \rbrace_{k=1}^{\infty} \subseteq \mathbb{N} $ such that for each $l \in \lbrace 1, \dots N \rbrace ,$ 
	$$s- \lim_{n \rightarrow \infty}D_{k}= s- \lim_{n \rightarrow \infty}G_{k}^{(l)}=P_{m}$$ 
	and the conditions  \eqref{cond22} and \eqref{cond222} hold.
	
	In particular, if $W_{1}, \dots , W_{N}$ satisfy the conditions \eqref{cond2}, \eqref{cond22} and \eqref{cond222}, then  $T_{U,W_{1}}^{r_{1}}, \dots , T_{U,W_{N}}^{r_{N}}$ are disjoint topologically transitive in $(B(H), SOT)$ for every unitary operator U.
\end{corollary}	

We denote the dual of $(B(\mathcal{H}),{\rm SOT})$ equipped with the weak-star topology  by $B(\mathcal{H})^*$.
If $\varphi\in B(\mathcal{H})^*$ and $D\in B(\mathcal{H})$, then we define 
$M_D\varphi\in B(\mathcal{H})^*$ by 
$M_D\varphi(F):=\varphi(DF)$ for all $F\in B(\mathcal{H})$. If $T_{U,W}$ is considered as an operator on $B(\mathcal{H})$, then we have
$T_{U,W}^*(\varphi)=\varphi_W\circ\widetilde{U}$
and 
$S_{U,W}^*(\varphi)=\varphi_{W^{-1}}\circ\widetilde{U}^{-1}$, where 
for each $D\in B(\mathcal{H})$,  $\widetilde{D}\in B(B(\mathcal{H}))$ is defined by $\widetilde{D}(F):=FD$ for all $F\in B(\mathcal{H})$.

In the following result, we give some sufficient conditions for a sequence of dual operators on $B(\mathcal{H})^*$ to be disjoint topologically transitive.
\begin{proposition}\label{thm33}
	Let $W_1,\ldots,W_N$  be invertible bounded linear operators on $\mathcal H$, and $U$ be a unitary operator in $B(\mathcal H)$.
	For each $k\in\mathbb{N}$ denote the operator $T_{U,W_k}$ on $B(\mathcal H)$ by $T_k$. Also, assume that $\{r_k\}_{k=1}^N\subseteq\mathbb{N}$ such that $0<r_1<r_2<\ldots<r_N$.  Then, ${\rm (ii)}$ implies ${\rm (i)}$, where
	\begin{itemize}
		\item [(i)] $T_{1}^{\ast ^{r_1}},\ldots,T_{N}^{\ast ^{r_N}}$ are d-topologically transitive on $B(\mathcal H)^*$.
		\item [(ii)] For every $n\in\mathbb{N}$ there exist sequences of operators $\{D_k\}$, $\{G_k^{(n)}\}$, $\ldots$, $\{G_k^{(N)}\}$ such that $s-\lim_{k\rightarrow\infty} D_k=s-\lim_{k\rightarrow\infty}G_k^{(l)}=P_n$ for all $l\in\{1,\ldots,N\}$, and 
		$$\lim_{k\longrightarrow \infty}\bigl\|D_k W_l^{r_ln_k}\bigr\|=\lim_{k\longrightarrow \infty}\bigl\|G^{(l)}_k W_l^{-r_ln_k}\bigr\|=0$$
		for all $l\in\{1,\ldots,N\}$. Moreover, for each distinct elements $s,\ l\in\{1,\ldots,N\}$, we have 
		$$\lim_{k\longrightarrow \infty}\bigl\|G^{(l)}_k W_s^{-r_sn_k}W_l^{r_ln_k}\bigr\|=0.$$
	\end{itemize}
\end{proposition}
\begin{proof}
	Let (ii) hold. Let $O,V_1,\ldots,V_N$ be nonempty $w^*$-open subsets of $B(\mathcal H)^*$. Pick $\psi\in O$, $\varphi_1\in V_1,\ldots,\ \varphi_N\in V_N$. We have $M_{P_n}\psi\rightarrow\psi$ and $M_{P_n}\varphi_l\rightarrow\varphi_l$ as $n\rightarrow\infty$ in $w^*$-topology for all $l\in\{1,\ldots,N\}$, because $P_nF\rightarrow F$ in the strong topology for each $F\in B(\mathcal{H})$. 
	Therefore, there is an $n\in\mathbb{N}$ such that 
	$M_{P_n}\psi\in O$ and $M_{P_n}\varphi_l\in V_l$ for all $l\in\{1,\ldots,N\}$.


	For each $k\in\mathbb{N}$, put
	$$\eta_k:=M_{P_nD_k}\psi+\sum_{l=1}^N S_{l}^{*^{r_ln_k}}(M_{P_nG_k^{(l)}}\varphi_l).$$
	After performing some simple calculations, we get that
	$\eta_k\rightarrow M_{P_n}\psi$ and $T_{l}^{*^{r_ln_k}}(\eta_k)\rightarrow M_{P_n}\varphi_l$ in $w^*$-topology as $k\rightarrow\infty$. This completes the proof.
\end{proof}

\begin{corollary}
	
	If the operators $W_{1}, \dots , W_{N} $ satisfy the conditions \eqref{cond2}, \eqref{cond22} and \eqref{cond222}, then $T_{U,W_{1}^{*}}^{*r_{1}} , \dots , T_{U,W_{N}^{*}}^{*r_{N}} $ are $d-$topologically transitive on $B(H)^{*} $ for every unitary operator $U.$
\end{corollary}

\begin{proof}
	If the conditions \eqref{cond2}, \eqref{cond22} and \eqref{cond222} are satisfied for the operators $W_{1}, \dots , W_{N}, $ by passing to the adjoints it is not hard to see that the conditions of Theorem \ref{thm33} will be satisfied for the operators $W_{1}^{*}, \dots W_{N}^{*}.$
\end{proof}


\begin{remark} \label{rem 2.16}
	If $\ell_2(\mathcal{A})$ denotes the standard Hilbert module over a $C^*$-algebra $\mathcal{A}$, then Theorem \ref{thm32} and Corollary \ref{coro210} can be transferred directly to this case with replacing $B_0(\mathcal{H})$  by the set of all compact $\mathcal{A}$-linear adjointable operators on $\ell_2(\mathcal{A})$, and $B(\mathcal{H})$ by $B^{a} (l_{2}(\mathcal{A}))$ where $B^{a} (l_{2}(\mathcal{A}))$ denotes the set of all bounded $\mathcal{A}$-linear adjointable operators on $\ell_2(\mathcal{A})$.	
	This is because the set 
	$$\lbrace P_{m}F \text{ } \vert \text{ } m \in \mathbb{N} \text{ and } F \in B^{a} (l_{2}(\mathcal{A})) \rbrace $$ 
	is dense in the set of all compact adjointable $\mathcal{A}-$linear operators on $l_{2} (\mathcal{A}) $ in the operator norm, if $\mathcal{A}$ is a unital $C^{*}-$algebra, and is also dense in $B^{a} (l_{2}(\mathcal{A})) $ in the strong operator topology if $\mathcal{A}$ is a general $C^{*}-$algebra. For more details, about $l_{2}(\mathcal{A}) ,$ see \cite{MT}.
\end{remark}

\begin{example}
	Let $\mathcal{A} $ be a non-unital $C^{*}-$algebra and suppose that $\lbrace p_{k} \rbrace_{ k \in \mathbb{N}} $ $ \lbrace g_{k}^{(1)}\rbrace_{k} , \dots , \lbrace g_{k}^{(N)}\rbrace_{k} $ are approximate units in $\mathcal{A} .$ Let $\lbrace w_{1}^{(j)} \rbrace_{j}, \dots ,  ,$ $\lbrace w_{N}^{(j)} \rbrace_{j} $ be sequences of invertible elements in $\mathcal{A} $ such that $\parallel w_{l}^{(j)} \parallel , \parallel w_{l}^{(j)-1} \parallel < M $ for all $j \in \mathbb{N} ,$ $l \in \lbrace 1, \dots , N \rbrace $ and some $M > 0 .$ Suppose in addition that for every $l \in \lbrace 1, \dots , N \rbrace $ we have that 
	$$\lim_{n \rightarrow \infty} \parallel w_{l}^{(j)r_{l}n} p_{k} \parallel = 0 \text{ and } \lim_{n \rightarrow \infty} \parallel w_{l}^{(j)-r_{l}{n}} g_{k}^{(l)} \parallel = 0 $$ for all $j,k \in \mathbb{N} .$ 
	
	Moreover, assume that 
	$$\lim_{n \rightarrow \infty} \parallel w_{l}^{(j)r_{l}n}  w_{s}^{(j)-r_{s}n}  g_{k}^{(s)} \parallel = 0 \text{ (*) }$$ 
	whenever $ s \neq l.$ For $l \in \lbrace 1, \dots , N \rbrace ,$ put $W_{l} $ to be the operator on $l_{2} ( \mathcal{A} ) $ given by $W_{l} ( \lbrace x_{j} \rbrace_{j} ) = \lbrace w_{l}^{(j)} x_{j} \rbrace $ for every $ \lbrace x_{j} \rbrace \in l_{2}(\mathcal{A}) .$ It follows then that $W_{l}, W_{l}^{-1} \in B^{a}(l_{2}(\mathcal{A})) $ for each $l \in \lbrace 1, \dots , N \rbrace .$ 
	
	If $\varphi_{k}^{(p)}, \varphi_{k}^{(g_{1})} , \dots ,  \varphi_{k}^{(g_{N})}$ are the operators on $l_{2} (\mathcal{A})$ given by $\varphi_{k}^{(p)}(\lbrace x_{j} \rbrace_{j} )= \lbrace p_{k} x_{j} \rbrace_{j} ,$ 
	$\varphi_{k}^{(g_{l})}(\lbrace x_{j} \rbrace_{j} ) = \lbrace g_{l}^{(j)} x_{j} \rbrace_{j} $ for $l \ \in \lbrace 1, \dots , N \rbrace $ and all $(\lbrace x_{j} \rbrace_{j} ) \in l_{2} (\mathcal{A}) ,$ then clearly these operators are bounded adjointable $\mathcal{A}-$ linear operators on $l_{2} (\mathcal{A}) .$ Given $m \in \mathbb{N} ,$ set $D_{k} = \varphi_{k}^{(p) } P_{m} $ and $G_{k}^{(l)} = \varphi_{k}^{(g_{l})} P_{m} $ for all $k \in \mathbb{N} $ and $  l  \in \lbrace 1, \dots , N \rbrace.$ It is not hard to see that 
	$$s - \lim_{k \rightarrow \infty} D_{k} = s - \lim_{k \rightarrow \infty} G_{k}^{(l)}=P_{m} .$$ 
	In addition, it also follows that 
	$$ \lim_{n \rightarrow \infty} \parallel W_{l}^{r_{l}n} D_{k} \parallel =  \lim_{n \rightarrow \infty} \parallel W_{l}^{-r_{l}n} G_{k}^{(l)} \parallel=0  $$ 
	for all $k \in \mathbb{N} $ and $l \in \lbrace 1, \dots , N \rbrace .$ Moreover due to $(^{*})$ we get that 
	$$\lim_{n \rightarrow \infty} \parallel W_{l}^{r_{l}n}  W_{s}^{-r_{s}n} G_{k}^{(l)} \parallel=0 $$ 
	whenever $s \neq l.$ Hence, it is not hard to construct a strictly increasing sequence $\lbrace n_{k} \rbrace_{k} \subseteq \mathbb{N} $ such that
	the conditions \eqref{cond22} and \eqref{cond222} hold with respect to that sequence. Then, the corresponding operators $T_{U,W_{1}}^{r_{1}}, \dots T_{U,W_{N}}^{r_{N}} $ are disjoint topologically transitive in the strong operator topology for every unitary operator $U $ on $l_{2}(\mathcal{A}) .$
	
	As a concrete example, we may let $\mathcal{A}=B_{0}(H) .$ If $\mathcal{A}=B_{0}(H) ,$ then for $j=0 ,$ we may let $w_{1}^{(0)} , w_{2}^{(0)}$ be the operators from Example \ref{examp 024} whereas for $j \neq 0 ,$ we may let $w_{1}^{(j)} , w_{2}^{(j)} $ be the operators on $H$ defined by
	
	$$w_{1}^{(j)} e_{i}=\begin{cases}
	(2+\dfrac{1}{\vert j \vert} ) e_{i}   & \text{  for } i<0 \,\\
	\dfrac{1}{2+\dfrac{1} {\vert j \vert} } e_{i}  & \text{  for } i \geq 0  \,\\
	\end{cases}  $$ 
	
	$$w_{2}^{(j)} e_{i}=\begin{cases}
	(3+\dfrac{1}{\vert j \vert} ) e_{i}   & \text{  for } i<0 \,\\
	\dfrac{1}{3+\dfrac{1} {\vert j \vert} } e_{i}  & \text{  for, } i \geq 0  \,\\
	\end{cases}  $$
	where $\lbrace e_{i} \rbrace_{i \in \mathbb{Z}} $ is an orthonormal basis for $H.$
	Then for each $k \in \mathbb{N} $ we set $p_{k} $ to be the orthonormal projection onto $ Span \lbrace e_{-k}, \dots , e_{k} \rbrace  $ and we put $g_{k}^{(1)}=p_{k} .$ Moreover, we choose an $r_{1} \in \mathbb{N} $ and we let $r_{2}=2r_{1} .$ By the same arguments as in Example \ref{examp 024} we can deduce that the conditions of Corollary   \ref{coro210} are satisfied for the corresponding operators $W_{1}^{r_{1}}$ and $W_{2}^{r_{2}}.$
\end{example}

\begin{remark}
	The equivalent conditions for dense d-hypercyclicity (or equivalently, disjoint topological transitivity) in Theorem \ref{thm32}  can be given in the situation where $B_0(\mathcal{H})$ is replaced by an arbitrary non-unital $C^*$-algebra $\mathcal{A}$, and the set of all finite rank orthogonal projections on $\mathcal{H}$ is replaced by  an approximate unit in $\mathcal{A}$ consisting of projections. See \cite[Remark 3]{iv2} for some details.
\end{remark}


\vspace{.1in}
\end{document}